\documentclass[leqno]{article}%
\usepackage{eurosym}
\usepackage{amsmath}
\usepackage{amsfonts}
\usepackage{amssymb}
\usepackage{graphicx}
\usepackage[colorlinks=true, pdfstartview=FitV, linkcolor=blue,
citecolor=blue, urlcolor=blue]{hyperref}
\usepackage{graphicx,color}
\usepackage[compress]{cite}%
\providecommand{\U}[1]{\protect\rule{.1in}{.1in}}
\newtheorem{theorem}{Theorem}

\newtheorem{corollary}{Corollary}

\newtheorem{lemma}{Lemma}

\newenvironment{proof}[1][Proof]{\noindent\textbf{#1.} }{\ \rule{0.5em}{0.5em}}
\allowdisplaybreaks
\begin{document}

\title{Reciprocity formulas for Hall-Wilson-Zagier type Hardy--Berndt sums }
\author{M\"{u}m\"{u}n CAN\\{\small Department of Mathematics, Akdeniz University, 07058-Antalya, Turkey }\\{\small E-mail: mcan@akdeniz.edu.tr}}
\date{}
\maketitle

\begin{abstract}
In this paper, we introduce vast generalizations of the Hardy--Berndt sums.
They involve higher-order Euler and/or Bernoulli functions, in which the
variables are affected by certain linear shifts. By employing the Fourier
series technique we derive linear relations for these sums. In particular,
these relations yield reciprocity formulas for Carlitz, Rademacher,
Mikol\'{a}s and Apostol type generalizations of the Hardy--Berndt sums, and
give rise to generalizations for some Goldberg's three-term relations. We also
present an elementary proof for the Mikol\'{a}s' linear relation and a
reciprocity formula in terms of the generation function.



\end{abstract}
\section{Introduction}

For integers $a$ and $c$ with $c>0$, the classical Dedekind sum $s(a,c)$ is
defined by
\begin{equation}
s(a,c)=\sum\limits_{v=1}^{c-1}\left(  \left(  \frac{v}{c}\right)  \right)
\left(  \left(  \frac{av}{c}\right)  \right)  , \label{ds}%
\end{equation}
where%
\[
\left(  \left(  x\right)  \right)  =%
\begin{cases}
x-\left[  x\right]  -1/2, & \text{if }x\not \in \mathbb{Z}\text{,}\\
0, & \text{if }x\in\mathbb{Z}\text{,}%
\end{cases}
\]
with $[x]$\ being the largest integer $\leq x$.

Analogous to Dedekind sum there are six arithmetic sums, known as Hardy sums
or Hardy--Berndt sums, defined by
\begin{align*}
&  S(a,c)={\sum\limits_{\mu=1}^{c-1}}\left(  -1\right)  ^{\mu+1+\left[
a\mu/c\right]  }, &  &  s_{1}\left(  a,c\right)  =\sum\limits_{\mu=1}%
^{c-1}\left(  -1\right)  ^{\left[  a\mu/c\right]  }\left(  \left(  \frac{\mu
}{c}\right)  \right)  ,\\
&  s_{2}\left(  a,c\right)  =\sum\limits_{\mu=1}^{c-1}\left(  -1\right)
^{\mu}\left(  \left(  \frac{a\mu}{c}\right)  \right)  \left(  \left(
\frac{\mu}{c}\right)  \right)  , &  &  s_{3}(a,c)={\sum\limits_{\mu=1}^{c-1}%
}\left(  -1\right)  ^{\mu}\left(  \left(  \dfrac{a\mu}{c}\right)  \right)  ,\\
&  s_{4}(a,c)={\sum\limits_{\mu=1}^{c-1}}\left(  -1\right)  ^{\left[
a\mu/c\right]  }, &  &  s_{5}(a,c)={\sum\limits_{\mu=1}^{c-1}}\left(
-1\right)  ^{\mu+\left[  a\mu/c\right]  }\left(  \left(  \dfrac{\mu}%
{c}\right)  \right)  .
\end{align*}

Dedekind sum appears in the transformation formulas of the logarithms of the
Dedekind eta-function, while Hardy--Berndt sums appear in the transformation
formulas of the logarithms of the classical theta functions \cite{b7,g}.
Goldberg \cite{g} showed that Hardy--Berndt sums also arise in the study on
the Fourier coefficients of the reciprocals of the classical theta functions
and in the theory of $r_{m}(n)$, the number of representations of $n$ as a sum
of $m$ integral squares. Moreover, Dedekind sum and its generalizations occur
in various areas such as topology \cite{HiZ,z}, algebraic geometry
\cite{Po,Ur}, combinatorial geometry \cite{BeR,Mo} and algorithmic complexity
\cite{Kn}.

One of the most important properties of these sums is their reciprocity
formula: it plays a key role in proving a bias phenomena \cite{axz},
distribution properties \cite{h,m2} and unboundedness \cite{can1,m2,r3} of the
sums. Let $a$ and $c$ be coprime positive integers. Then,\textbf{\ }%
\[
s(a,c)+s(c,a)=-\frac{1}{4}+\frac{1}{12}\left(  \frac{a}{c}+\frac{c}{a}%
+\frac{1}{ac}\right)
\]
and \cite{b7,g}%
\begin{align}
S(a,c)+S(c,a)  &  =1,\text{ if }a+c\text{ is odd,}\label{hb-0}\\
s_{1}(a,c)-2s_{2}(c,a)  &  =\frac{1}{2}-\frac{1}{2}\left(  \frac{1}{ac}%
+\frac{c}{a}\right)  ,\text{ if }a\text{ is even,}\label{hb-12}\\
2s_{3}(a,c)-s_{4}(c,a)  &  =1-\frac{a}{c},\text{ if }c\text{ is odd,}%
\label{hb-34}\\
s_{5}\left(  a,c\right)  +s_{5}\left(  c,a\right)   &  =\frac{1}{2}-\frac
{1}{2ac},\text{ if }a+c\text{ is even.} \label{hb-5}%
\end{align}

Let $\mathcal{B}_{n}\left(  x\right)  =B_{n}\left(  x-\left[  x\right]
\right)  $ is the $n$th Bernoulli function with $B_{n}(x)$ being the $n$th
Bernoulli polynomial \cite[p. 804]{as}. One of the various generalizations of
the Dedekind sum, due to Hall, Wilson and Zagier \cite{hwz}, is\textbf{\ }%
\begin{equation}
s_{p,q}\binom{a\ b\ c}{x\ y\ z}=\sum\limits_{v=0}^{c-1}\overline{B}_{p}\left(
a\frac{v+z}{c}-x\right)  \overline{B}_{q}\left(  b\frac{v+z}{c}-y\right)  ,
\label{hwz-ds}%
\end{equation}
where $\overline{B}_{p}\left(  x\right)  =\mathcal{B}_{p}\left(  x\right)  $
for $p\not =1,$ and $\overline{B}_{1}\left(  x\right)  =\left(  \left(
x\right)  \right)  .$ This sum contains generalized Dedekind sums
previously-defined by Carlitz \cite{car3} (see also \cite{t})
\begin{equation}
s_{p}(a,c:x,y)=\sum\limits_{v=0}^{c-1}\mathcal{B}_{p}\left(  a\frac{v+y}%
{c}+x\right)  \left(  \left(  \frac{v+y}{c}\right)  \right)  , \label{ct-ds}%
\end{equation}
Rademacher \cite{r2}
\begin{equation}
s(a,c:x,y)=\sum\limits_{v=0}^{c-1}\left(  \left(  a\frac{v+y}{c}+x\right)
\right)  \left(  \left(  \frac{v+y}{c}\right)  \right)  , \label{r-ds}%
\end{equation}
Mikol\'{a}s \cite{mi} (see also \cite{car1})
\begin{equation}
s_{p,q}(a,b,c)=\sum\limits_{v=0}^{c-1}\mathcal{B}_{p}\left(  \frac{av}%
{c}\right)  \mathcal{B}_{q}\left(  \frac{bv}{c}\right)  \label{mi-ds}%
\end{equation}
and Apostol \cite{a1}
\begin{equation}
s_{p}(a,c)=\sum\limits_{v=0}^{c-1}\overline{B}_{p}\left(  \frac{av}{c}\right)
\left(  \left(  \frac{v}{c}\right)  \right)  . \label{a-ds}%
\end{equation}
It should be mentioned that Hall and Wilson \cite{hw} classified all linear
relations (reciprocity formulas) for the sums $s_{p,q}(a,b,c)$ and
$\ s_{p,q}(a,1,c),$ and it emerged that Mikol\'{a}s' relations form a complete
set \cite[Eq. (5.5)]{mi} (see also \cite[Eq. (8)]{hw}). Moreover, Hall, Wilson
and Zagier's reciprocity formula is in terms of the generating function
\[
\mathcal{G}\left(
\begin{matrix}
a & b & c\\
x & y & z\\
X & Y & Z
\end{matrix}
\right)  =\sum\limits_{p,q\geq0}\frac{1}{p!q!}s_{p,q}\binom{a\ b\ c}%
{x\ y\ z}\left(  X/a\right)  ^{p-1}\left(  Y/b\right)  ^{q-1}.
\]

\begin{theorem}
(\cite[Theorem]{hwz}) Let $a,$ $b,$ $c$ be pairwise coprime positive integers,
$x,$ $y,$ $z\in\mathbb{R}$, and let $X,$ $Y,$ $Z$ be nonzero variables such
that $X+Y+Z=0$. Then
\begin{align}
&  \mathcal{G}\left(
\begin{matrix}
a & b & c\\
x & y & z\\
X & Y & Z
\end{matrix}
\right)  +\mathcal{G}\left(
\begin{matrix}
b & c & a\\
y & z & x\\
Y & Z & X
\end{matrix}
\right)  +\mathcal{G}\left(
\begin{matrix}
c & a & b\\
z & x & y\\
Z & X & Y
\end{matrix}
\right) \nonumber\\
\  &  =%
\begin{cases}
-1/4, & \text{if }(x,y,z)\in(a,b,c)\mathbb{R}+\mathbb{Z}^{3},\\
0, & \text{otherwise. }%
\end{cases}
\nonumber
\end{align}

\end{theorem}

For further generalizations of the Dedekind sum, for example, see
\cite{bc,b6,cen,cck2,m4,rs,t,z}.

Several generalizations of the Hardy--Berndt sums have also been introduced
\cite{bc-c,cck,cd,ck,dc,lz,lg,m4,ys3,ys4}, some of which also obey reciprocity
formula \cite{bc-c,cck,ck,cd,dc}, and various properties have been studied
such as finite trigonometric and infinite series representations
\cite{bg,cck,g,ys2}, three-term relations \cite{g,ps,ys1}, distribution
properties \cite{m2}, unboundedness \cite{can1,m2}, Petersson--Knopp identity
\cite{lg,s} and mean value \cite{lz1,lz,pz}. However, to the author's
knowledge, generalizations of the Hardy--Berndt sums in the sense of
(\ref{hwz-ds}), (\ref{ct-ds}), (\ref{r-ds})\ and (\ref{mi-ds}), have not been studied.

By employing the Fourier series technique we demonstrate a formula for the product $\mathcal{B}_{p}\left(  X+Y\right)  \mathcal{B}_{q}\left(
Y\right)$ (see (\ref{0}) below), motivated by the second proof of \cite[Proposition]{hwz}. This
formula permits to easily produce linear relations for generalizations of the
aforementioned sums. This paper is concerned with the generalizations of the
Hardy--Berndt sums in the sense of (\ref{hwz-ds}). They involve higher-order
Euler and/or Bernoulli functions, so we call these sums higher-order
Hardy--Berndt sums. Reciprocity formulas for Carlitz (\ref{ct-ds}), Rademacher
(\ref{r-ds}), Mikol\'{a}s (\ref{mi-ds}) and Apostol (\ref{a-ds}) type
generalizations immediately follow from linear relations of higher-order
Hardy--Berndt sums. In addition, derived formulas give rise
to\ generalizations for some Goldberg's three-term relations \cite[Chapter
5]{g} (see also \cite{ps,ys1}).

All linear relations proved in this paper are derived from (\ref{0}). By
changing the parameters in (\ref{0}) and summing the resulting expressions, we
obtain all linear relations for generalized Hardy--Berndt sums in this paper
(see the proofs of Theorems \ref{rpS}, \ref{rps12}, \ref{rps543}). Moreover,
(\ref{0}) enables an elementary proof for the Mikol\'{a}s' linear relation
\cite[Eq. (5.5)]{mi} or \cite[Eq. (8)]{hw}. This will be presented in Section
\ref{mlr}. Besides that in Section \ref{rp-gf}, we offer a reciprocity formula
in terms of the generating function.

For instance, a generalization of the sum $S\left(  a,c\right)  ,$ introduced
here, is
\[
S_{p,q}\left(  a,b,c:x,y,z\right)  =\sum\limits_{\mu=0}^{c-1}\left(
-1\right)  ^{\mu}\mathcal{E}_{p-1}\left(  a\frac{\mu+z}{c}+x\right)
\mathcal{E}_{q-1}\left(  b\frac{\mu+z}{c}+y\right)  ,
\]
where $\mathcal{E}_{n}\left(  x\right)  $\ is the $n$th Euler function defined
by, for $0\leq x<1$ and $m\in\mathbb{Z}$,%
\begin{equation}
\mathcal{E}_{n}\left(  x\right)  =E_{n}\left(  x\right)  \text{ and
}\mathcal{E}_{n}\left(  x+m\right)  =\left(  -1\right)  ^{m}\mathcal{E}%
_{n}\left(  x\right)  ,\text{ }n\geq0, \label{e}%
\end{equation}
and $E_{n}\left(  x\right)  $ denotes the $n$th Euler polynomial \cite[p.
804]{as}.

We shall prove the following linear relation, which is analogue of the
Mikol\'{a}s' relation \cite[Eq. (5.5)]{mi}.\

\begin{theorem}
\label{rpS}Let $a$, $b$ and $c$ be pairwise coprime positive integers with
$a+b+c$ even. Then, for $p,q\geq1$ and $x,y,z\in\mathbb{R}$ we have
\begin{align}
&  a^{1-p}b^{1-q}S_{p,q}\left(  a,b,c:x,y,z\right) \label{rp-S}\\
&  =\sum\limits_{j=1}^{p}\binom{p-1}{j-1}a^{1-j}c^{1+j-p-q}S_{p+q-j,j}\left(
c,-a,b:z,x,y\right) \nonumber\\
&  \quad+\sum\limits_{h=1}^{q}\binom{q-1}{h-1}\left(  -1\right)  ^{h}%
b^{1-h}c^{1+h-p-q}S_{p+q-h,h}\left(  c,b,a:z,-y,x\right)  .\nonumber
\end{align}

\end{theorem}

In particular, if $p=q=1$, (\ref{rp-S}) gives the following three-term
relation
\begin{equation}
S_{1,1}\left(  a,b,c:x,y,z\right)  -S_{1,1}\left(  c,-a,b:z,x,y\right)
+S_{1,1}\left(  c,b,a:z,-y,x\right)  =0. \label{12}%
\end{equation}
For $x=y=z=0,$ (\ref{12}) reduces to Goldberg's three-term relation
\cite[Theorem 5.2]{g} (cf. \cite[Eqs. (3.2)--(3.4)]{ps})%
\begin{equation}
S\left(  a,b,c\right)  +S\left(  c,a,b\right)  +S\left(  c,b,a\right)  =1,
\label{12a}%
\end{equation}
where
\[
S\left(  a,b,c\right)  =\sum\limits_{\mu=1}^{c-1}\left(  -1\right)  ^{\mu
+1}\mathcal{E}_{0}\left(  \frac{a\mu}{c}\right)  \mathcal{E}_{0}\left(
\frac{b\mu}{c}\right)  =-S_{1,1}\left(  a,b,c:0,0,0\right)  +1,
\]
and $S_{1,1}\left(  c,-a,b:0,0,0\right)  =S\left(  c,a,b\right)  +1$ by the
reflection identity
\begin{equation}
\mathcal{E}_{n}\left(  -x\right)  =\left(  -1\right)  ^{n-1}\mathcal{E}%
_{n}\left(  x\right)  ,\text{ if }n\not =0\text{ or }x\not \in \mathbb{Z}%
\text{.} \label{7}%
\end{equation}

Moreover, (\ref{rp-S}) yields the following reciprocity formula for the sum
\[
S_{p}\left(  a,c:x,z\right)  =S_{p,1}\left(  a,1,c:x,0,z\right)
=\sum\limits_{\mu=0}^{c-1}\left(  -1\right)  ^{\mu}\mathcal{E}_{p-1}\left(
a\frac{\mu+z}{c}+x\right)  \mathcal{E}_{0}\left(  \frac{\mu+z}{c}\right)  ,
\]
which is Carlitz (\ref{ct-ds}) type generalization (without doubt also
Rademacher (\ref{r-ds}) and Apostol (\ref{a-ds}) type generalizations) of
$S\left(  a,c\right)  $.

\begin{corollary}
\label{cor-S}Let $a$ and $c$ be coprime positive integers with $a+c$ odd.
Then, for $p\geq1$ and $x,z\in\mathbb{R}$ we have
\begin{equation}
ac^{p}S_{p}\left(  a,c:x,z\right)  +ca^{p}S_{p}\left(  c,a:z,x\right)
=\sum\limits_{j=1}^{p}\binom{p-1}{j-1}a^{p+1-j}c^{j}\mathcal{E}_{p-j}\left(
z\right)  \mathcal{E}_{j-1}\left(  x\right)  . \label{15}%
\end{equation}

\end{corollary}

It is clear that (\ref{hb-0}) is immediate consequence of both (\ref{12a}) and
(\ref{15}).

\section{Main theorem}

In this section we will prove the following identity, which is the key to
derive linear relations for higher-order Hardy--Berndt sums.

\begin{theorem}
\label{BpBq}For $X,Y\in\mathbb{R}$ and $p,q\geq1,$ we have%
\begin{align}
&  \binom{p+q}{q}\mathcal{B}_{p}\left(  X+Y\right)  \mathcal{B}_{q}\left(
Y\right) \label{0}\\
&  =\sum\limits_{j=0}^{p+q}\binom{p+q}{j}\binom{p+q-1-j}{q-1}\mathcal{B}%
_{p+q-j}\left(  Y\right)  \mathcal{B}_{j}\left(  X\right) \nonumber\\
&  \quad+\sum\limits_{h=0}^{q}\binom{p+q}{h}\binom{p+q-1-h}{p-1}\left(
-1\right)  ^{h}\mathcal{B}_{p+q-h}\left(  X+Y\right)  \mathcal{B}_{h}\left(
X\right)  .\nonumber
\end{align}

\end{theorem}

Before proving this, we consider the special case $q=1.$ For $p\geq1$ and
$q=1,$ (\ref{0}) reduces to
\begin{align}
\left(  p+1\right)  \mathcal{B}_{p}\left(  X+Y\right)  \mathcal{B}_{1}\left(
Y\right)   &  =\sum\limits_{j=0}^{p+1}\binom{p+1}{j}\mathcal{B}_{p+1-j}\left(
Y\right)  \mathcal{B}_{j}\left(  X\right) \label{01}\\
&  \ +p\mathcal{B}_{p+1}\left(  X+Y\right)  -\left(  p+1\right)
\mathcal{B}_{p}\left(  X+Y\right)  \mathcal{B}_{1}\left(  X\right)  .\nonumber
\end{align}
This is nothing but Lemma 3 of Tak\'{a}cs \cite{t} and holds for $p\geq1$ and
any $X,$ $Y$. Taking $X\rightarrow-X$ and then $Y\rightarrow Y+X$ in
(\ref{01}) give%
\begin{align*}
\left(  r+1\right)  \mathcal{B}_{r}\left(  Y\right)  \mathcal{B}_{1}\left(
Y+X\right)   &  =\sum\limits_{j=0}^{r+1}\binom{r+1}{j}\mathcal{B}%
_{r+1-j}\left(  Y+X\right)  \mathcal{B}_{j}\left(  -X\right) \\
&  \quad+r\mathcal{B}_{r+1}\left(  Y\right)  -\left(  r+1\right)
\mathcal{B}_{r}\left(  Y\right)  \mathcal{B}_{1}\left(  -X\right)  .
\end{align*}
Here, using the reflection identity%
\begin{equation}
\mathcal{B}_{p}\left(  -x\right)  =\left(  -1\right)  ^{p}\mathcal{B}%
_{p}\left(  x\right)  ,\text{ if }p\not =1\text{ or }x\not \in \mathbb{Z},
\label{8}%
\end{equation}
according to $X\not \in \mathbb{Z}$ or $X\in\mathbb{Z}$, it is seen that
\begin{align}
\left(  r+1\right)  \mathcal{B}_{r}\left(  Y\right)  \mathcal{B}_{1}\left(
Y+X\right)   &  =\sum\limits_{j=0}^{r+1}\binom{r+1}{j}\left(  -1\right)
^{j}\mathcal{B}_{r+1-j}\left(  Y+X\right)  \mathcal{B}_{j}\left(  X\right)
\label{02}\\
&  \quad+r\mathcal{B}_{r+1}\left(  Y\right)  +\left(  r+1\right)
\mathcal{B}_{r}\left(  Y\right)  \mathcal{B}_{1}\left(  X\right)  ,\nonumber
\end{align}
which corresponds to (\ref{0}) for the case $p=1$ and $q\geq1$.

So, it is enough to show that (\ref{0}) is true for\textbf{\ }$p\geq2$ and
$q\geq2.$ For this, we recall the Fourier series representation
\begin{equation}
\mathcal{B}_{p}\left(  x\right)  =-\frac{p!}{(2\pi i)^{p}}\sum
\limits_{\substack{m=-\infty\\m\neq0}}^{\infty}\frac{e^{2\pi imx}}{m^{p}%
},\text{ (see \cite[p. 805]{as})} \label{f1}%
\end{equation}
where $x\in\mathbb{R}$ if $p>1,$ and $x\not \in \mathbb{Z}$ if $p=1.$

\begin{proof}
[Proof of Theorem \ref{BpBq}]Let $A_{p}=-p!/(2\pi i)^{p}.$ From (\ref{f1}), we
have
\begin{align}
\frac{\mathcal{B}_{p}\left(  X+Y\right)  \mathcal{B}_{q}\left(  Y\right)
}{A_{p}A_{q}}  &  =\sum\limits_{\substack{m=-\infty\\m\neq0}}^{\infty}%
\sum\limits_{\substack{n=-\infty\\n\neq0}}^{\infty}\frac{e^{2\pi i\left(
\left(  m+n\right)  \left(  X+Y\right)  -nX\right)  }}{m^{p}n^{q}}\label{1}\\
&  =\sum\limits_{m+n=0}{}\hspace{-0.07in}^{^{\prime}}\ \frac{e^{-2\pi inX}%
}{m^{p}n^{q}}+\sum\limits_{m+n\not =0}{}\hspace{-0.07in}^{^{\prime}}%
\ \frac{e^{2\pi i\left(  \left(  m+n\right)  \left(  X+Y\right)  -nX\right)
}}{m^{p}n^{q}}.\nonumber
\end{align}
Here and in the sequel, we write
\[
\sum\ \hspace{-0.05in}^{^{\prime}}\text{ for }\sum\limits_{\substack{m=-\infty
\\m\neq0}}^{\infty}\sum\limits_{\substack{n=-\infty\\n\neq0}}^{\infty
}=\underset{M\rightarrow\infty}{\lim}\sum\limits_{\substack{m=-M \\m\neq
0}}^{M}\underset{N\rightarrow\infty}{\lim}\sum\limits_{\substack{n=-N
\\n\neq0}}^{N}.
\]
We make the substitution $n=r-m$\ to write (\ref{1}) as%
\begin{equation}
\frac{\mathcal{B}_{p}\left(  X+Y\right)  \mathcal{B}_{q}\left(  Y\right)
}{A_{p}A_{q}}=\left(  -1\right)  ^{q}\frac{\mathcal{B}_{p+q}\left(  X\right)
}{A_{p+q}}+\sum\limits_{m\not =r}{}^{^{\prime}}\ \frac{e^{2\pi i\left(
rY+mX\right)  }}{m^{p}\left(  r-m\right)  ^{q}}. \label{2}%
\end{equation}
We now utilize the following partial fractions in (\ref{2}):
\[
\frac{1}{x^{p}\left(  1-x\right)  ^{q}}=\sum\limits_{j=0}^{p-1}\frac
{\alpha_{p-j}}{x^{p-j}}+\sum\limits_{l=0}^{q-1}\frac{\beta_{q-l}}{\left(
1-x\right)  ^{q-l}},
\]
where $\alpha_{p-j}=\binom{q+j-1}{j},$ $0\leq j<p$\ and $\beta_{q-l}%
=\binom{p+l-1}{l},$ $0\leq l<q.$\ We then deduce
\begin{align}
&  \sum\limits_{m\not =r}{}^{^{\prime}}\ \frac{e^{2\pi i\left(  rY+mX\right)
}}{m^{p}\left(  r-m\right)  ^{q}}\label{3}\\
&  =\sum\limits_{j=0}^{p-1}\alpha_{p-j}\sum\limits_{m\not =r}{}^{^{\prime}%
}\ \frac{e^{2\pi i\left(  rY+mX\right)  }}{m^{p-j}r^{q+j}}+\sum\limits_{l=0}%
^{q-1}\beta_{q-l}\sum\limits_{m\not =r}{}^{^{\prime}}\ \frac{e^{2\pi i\left(
rY+mX\right)  }}{r^{p+l}\left(  r-m\right)  ^{q-l}}.\nonumber
\end{align}
It is not difficult to see that%
\begin{equation}
\sum\limits_{m\not =r}{}^{^{\prime}}\ \frac{e^{2\pi i\left(  rY+mX\right)  }%
}{m^{p-j}r^{q+j}}=\frac{\mathcal{B}_{p-j}\left(  X\right)  \mathcal{B}%
_{q+j}\left(  Y\right)  }{A_{p-j}A_{q+j}}-\frac{\mathcal{B}_{p+q}\left(
X+Y\right)  }{A_{p+q}}. \label{4}%
\end{equation}
Similar to (\ref{2}), we have%
\begin{align}
\frac{\mathcal{B}_{q}\left(  X\right)  \mathcal{B}_{p}\left(  X+Y\right)
}{A_{q}A_{p}}  &  =\sum\limits_{\substack{k=-\infty\\k\neq0}}^{\infty}%
\sum\limits_{\substack{r=-\infty\\r\neq0}}^{\infty}\frac{e^{2\pi i\left(
kX+r\left(  X+Y\right)  \right)  }}{k^{q}r^{p}}\nonumber\\
&  =\left(  -1\right)  ^{q}\frac{\mathcal{B}_{p+q}\left(  Y\right)  }{A_{p+q}%
}+\sum\limits_{m\not =r}{}^{^{\prime}}\ \frac{e^{2\pi i\left(  rY+mX\right)
}}{\left(  m-r\right)  ^{q}r^{p}},\nonumber
\end{align}
which yields
\begin{equation}
\sum\limits_{m\not =r}{}^{^{\prime}}\ \frac{e^{2\pi i\left(  rY+mX\right)  }%
}{r^{p+l}\left(  r-m\right)  ^{q-l}}=\left(  -1\right)  ^{q-l}\frac
{\mathcal{B}_{p+l}\left(  X+Y\right)  \mathcal{B}_{q-l}\left(  X\right)
}{A_{p+l}A_{q-l}}-\frac{\mathcal{B}_{p+q}\left(  Y\right)  }{A_{p+q}}.
\label{5}%
\end{equation}
Combining (\ref{2}), (\ref{3}), (\ref{4}), (\ref{5}) and using the fact
\[
\sum\limits_{j=0}^{p-1}\alpha_{p-j}=\sum\limits_{j=0}^{p-1}\binom{q+j-1}%
{j}=\binom{q+p-1}{p-1},
\]
after simplifications, we arrive at
\begin{align}
&  \binom{p+q}{q}\mathcal{B}_{p}\left(  X+Y\right)  \mathcal{B}_{q}\left(
Y\right) \label{03}\\
&  =\binom{p+q-1}{q-1}\mathcal{B}_{p+q}\left(  Y\right)  +\sum\limits_{j=1}%
^{p}\binom{p+q}{j}\binom{p+q-1-j}{q-1}\mathcal{B}_{j}\left(  X\right)
\mathcal{B}_{p+q-j}\left(  Y\right) \nonumber\\
&  \quad-\left(  -1\right)  ^{q}\mathcal{B}_{p+q}\left(  X\right)
+\binom{p+q-1}{p-1}\mathcal{B}_{p+q}\left(  X+Y\right) \nonumber\\
&  \quad+\sum\limits_{l=0}^{q-1}\binom{p+q}{q-l}\binom{p+l-1}{p-1}\left(
-1\right)  ^{q-l}\mathcal{B}_{p+l}\left(  X+Y\right)  \mathcal{B}_{q-l}\left(
X\right)  .\nonumber
\end{align}
This is equivalent to (\ref{0}) for $p\geq2$ and $q\geq2$ since
\[
\binom{p+q-1-j}{q-1}=\left\{
\begin{array}
[c]{ll}%
0, & \text{if }p+1\leq j\leq p+q-1,\\
\left(  -1\right)  ^{q-1}, & \text{if }j=p+q.
\end{array}
\right.
\]
Hence, (\ref{0}) holds for $p\geq1$ and $q\geq1$ from (\ref{01}), (\ref{02})
and (\ref{03}).
\end{proof}

\section{Higher-order Hardy--Berndt sums}

As mentioned in the introductory section, linear relations for higher-order
Hardy--Berndt sums are deduced from (\ref{0}) with the following
multiplication formulas: if $r$ is positive integer, then (see \cite[p.
804]{as})%

\begin{equation}
\mathcal{B}_{n}\left(  x\right)  =r^{n-1}\sum\limits_{v=0}^{r-1}%
\mathcal{B}_{n}\left(  \frac{x+v}{r}\right)  ,\text{ }n\geq0, \label{17}%
\end{equation}
if $r$ is odd positive integer, then
\begin{equation}
\mathcal{E}_{n}\left(  x\right)  =r^{n}\sum\limits_{v=0}^{r-1}\left(
-1\right)  ^{v}\mathcal{E}_{n}\left(  \frac{x+v}{r}\right)  ,\text{ }n\geq0,
\label{ee}%
\end{equation}
if $r$ is even positive integer, then
\begin{equation}
\mathcal{E}_{n-1}\left(  x\right)  =-\frac{2}{n}r^{n-1}\sum\limits_{v=0}%
^{r-1}\left(  -1\right)  ^{v}\mathcal{B}_{n}\left(  \frac{x+v}{r}\right)
,\text{ }n\geq1. \label{be}%
\end{equation}

\subsection{Generalizations of the sums $s_{1}(a,c) $ and $s_{2}(a,c)$}

Let $a$, $b$ and $c$ be positive integers with $c$ even. We first set
\[
X=\frac{v+x}{a}-\frac{u+y}{b},\text{ }Y=\frac{\mu+z}{c}+\frac{u+y}{b}%
\]
in (\ref{0}) and multiply both sides with $\left(  -1\right)  ^{\mu}.$\ Then
summing over $v\left(  \operatorname{mod}a\right)  ,$ $u\left(
\operatorname{mod}b\right)  $ and $\mu\left(  \operatorname{mod}c\right)  ,$
with the use of (\ref{17}) and (\ref{be}), we deduce that\textbf{\ }%
\begin{align*}
&  a^{1-p}b^{1-q}\sum\limits_{\mu=0}^{c-1}\left(  -1\right)  ^{\mu}%
\mathcal{B}_{p}\left(  a\frac{\mu+z}{c}+x\right)  \mathcal{B}_{q}\left(
b\frac{\mu+z}{c}+y\right) \\
&  \ =-\frac{q}{2}\sum\limits_{j=0}^{p}\binom{p}{j}\frac{1}{a^{j-1}%
c^{p+q-j-1}}\sum\limits_{u=0}^{b-1}\mathcal{E}_{p+q-j-1}\left(  c\frac{u+y}%
{b}+z\right)  \mathcal{B}_{j}\left(  -a\frac{u+y}{b}+x\right) \\
&  \quad-\frac{p}{2}\sum\limits_{h=0}^{q}\binom{q}{h}\frac{\left(  -1\right)
^{h}}{b^{h-1}c^{p+q-h-1}}\sum\limits_{v=0}^{a-1}\mathcal{E}_{p+q-h-1}\left(
c\frac{v+x}{a}+z\right)  \mathcal{B}_{h}\left(  b\frac{v+x}{a}-y\right)  ,
\end{align*}
where we have used that the sum over $\mu$ is zero for $j=p+q.$ The
observations%
\[
\sum\limits_{\mu=0}^{c-1}\mathcal{E}_{0}\left(  \frac{a\mu}{c}\right)
\mathcal{B}_{1}\left(  \frac{\mu}{c}\right)  =s_{1}\left(  a,c\right)
+\mathcal{E}_{0}\left(  0\right)  \mathcal{B}_{1}\left(  0\right)
\]
and%
\[
\sum\limits_{\mu=0}^{c-1}\left(  -1\right)  ^{\mu}\mathcal{B}_{1}\left(
\frac{a\mu}{c}\right)  \mathcal{B}_{1}\left(  \frac{\mu}{c}\right)
=s_{2}\left(  a,c\right)  +\mathcal{B}_{1}\left(  0\right)  \mathcal{B}%
_{1}\left(  0\right)  ,
\]
suggest to define
\begin{align*}
S_{p,q}^{\left(  1\right)  }\left(  a,b,c:x,y,z\right)   &  =\sum
\limits_{\mu=0}^{c-1}\mathcal{E}_{p-1}\left(  a\frac{\mu+z}{c}+x\right)
\mathcal{B}_{q}\left(  b\frac{\mu+z}{c}+y\right)  ,\\
S_{p,q}^{\left(  2\right)  }\left(  a,b,c:x,y,z\right)   &  =\sum
\limits_{\mu=0}^{c-1}\left(  -1\right)  ^{\mu}\mathcal{B}_{p}\left(
a\frac{\mu+z}{c}+x\right)  \mathcal{B}_{q}\left(  b\frac{\mu+z}{c}+y\right)  .
\end{align*}
Hence, we have proved that these sums satisfy the following linear relation.

\begin{theorem}
\label{rps12}Let $a$, $b$ and $c$ be positive integers with $c$ even. Then,
for $p,q\geq1$ and $x,y,z\in\mathbb{R}$ we have\textbf{\ }
\begin{align}
&  a^{1-p}b^{1-q}S_{p,q}^{\left(  2\right)  }\left(  a,b,c:x,y,z\right)
\label{rp-s12}\\
&  \ =-\frac{q}{2}\sum\limits_{j=0}^{p}\binom{p}{j}a^{1-j}c^{1+j-p-q}%
S_{p+q-j,j}^{\left(  1\right)  }\left(  c,-a,b:z,x,y\right) \nonumber\\
&  \quad-\frac{p}{2}\sum\limits_{h=0}^{q}\binom{q}{h}\left(  -1\right)
^{h}b^{1-h}c^{1+h-p-q}S_{p+q-h,h}^{\left(  1\right)  }\left(
c,b,a:z,-y,x\right)  .\nonumber
\end{align}

\end{theorem}

While considering special cases, we need the following lemma.

\begin{lemma}
Let $a$ and $c$ be coprime positive integers with $a$ odd. Then,
\begin{align}
\sum\limits_{\mu=0}^{a-1}\mathcal{E}_{p}\left(  c\frac{\mu+x}{a}+z\right)   &
=a^{-p}\mathcal{E}_{p}\left(  az+cx\right)  ,\text{ if }c\text{ is
even,}\label{27}\\
\sum\limits_{\mu=0}^{a-1}\left(  -1\right)  ^{\mu}\mathcal{E}_{p}\left(
c\frac{\mu+x}{a}+z\right)   &  =a^{-p}\mathcal{E}_{p}\left(  az+cx\right)
,\text{ if }c\text{ is odd.} \label{11}%
\end{align}

\end{lemma}

\begin{proof}
If $c$ is even, then%
\begin{align*}
\sum\limits_{v=0}^{a-1}\mathcal{E}_{p}\left(  c\frac{v+x}{a}+z\right)   &
\overset{\left(  \text{\ref{be}}\right)  }{=}-\frac{2}{p+1}c^{p}%
\sum\limits_{j=0}^{c-1}\sum\limits_{v=0}^{a-1}\left(  -1\right)
^{j}\mathcal{B}_{p+1}\left(  \frac{v+x}{a}+\frac{j+z}{c}\right) \\
&  \overset{\left(  \text{\ref{17}}\right)  }{=}-\frac{2c^{p}a^{-p}}{p+1}%
\sum\limits_{j=0}^{c-1}\left(  -1\right)  ^{j}\mathcal{B}_{p+1}\left(
\frac{aj}{c}+\frac{az}{c}+x\right) \\
&  \ \ \mathcal{=}-\frac{2c^{p}a^{-p}}{p+1}\sum\limits_{j=0}^{c-1}\left(
-1\right)  ^{j}\mathcal{B}_{p+1}\left(  \frac{j}{c}+\frac{az}{c}+x\right)
,\text{%
\begin{tabular}
[c]{l}%
$\text{(}c\text{ even and}$\\
$\ a\text{ odd)}$%
\end{tabular}
}\\
&  \overset{\left(  \text{\ref{be}}\right)  }{=}a^{-p}\mathcal{E}_{p}\left(
az+cx\right)  \text{.}%
\end{align*}

Let $c$ be odd. We substitute $cv=\mu+ak_{\mu}$ and see that%
\begin{align*}
\sum\limits_{v=0}^{a-1}\left(  -1\right)  ^{v}\mathcal{E}_{p}\left(
\frac{cv+cx}{a}+z\right)   &  =\sum\limits_{\mu=0}^{a-1}\left(  -1\right)
^{\mu+ak_{\mu}}\mathcal{E}_{p}\left(  \frac{\mu+cx}{a}+z+k_{\mu}\right) \\
&  \overset{\text{(\ref{e})}}{=}\sum\limits_{\mu=0}^{a-1}\left(  -1\right)
^{\mu}\mathcal{E}_{p}\left(  \frac{\mu+cx}{a}+z\right)  ,\text{ }a\text{
odd}\\
&  \overset{\text{(\ref{ee})}}{=}a^{-p}\mathcal{E}_{p}\left(  cx+az\right)  .
\end{align*}

\end{proof}

Now we mention some special cases of (\ref{rp-s12}). Let $a$, $b$ and $c$ be
pairwise coprime positive integers with $c$ even. For $p=q=1$, we have%
\begin{align*}
&  S_{1,1}^{\left(  1\right)  }\left(  c,b,a:z,-y,x\right)  -S_{1,1}^{\left(
1\right)  }\left(  c,-a,b:z,x,y\right)  -2S_{1,1}^{\left(  2\right)  }\left(
a,b,c:x,y,z\right) \\
&  \ =\frac{a}{c}S_{2,0}^{\left(  1\right)  }\left(  c,-a,b:z,x,y\right)
+\frac{b}{c}S_{2,0}^{\left(  1\right)  }\left(  c,b,a:z,-y,x\right)  .
\end{align*}
Since\textbf{\ }%
\[
S_{2,0}^{\left(  1\right)  }\left(  c,b,a:z,-y,x\right)  =\sum\limits_{v=0}%
^{a-1}\mathcal{E}_{1}\left(  c\frac{v+x}{a}+z\right)  \overset{\left(
\text{\ref{27}}\right)  }{=}\frac{1}{a}\mathcal{E}_{1}\left(  az+cx\right)
\text{,}%
\]
we arrive at the three-term relation
\begin{align}
&  S_{1,1}^{\left(  1\right)  }\left(  c,b,a:z,-y,x\right)  -S_{1,1}^{\left(
1\right)  }\left(  c,-a,b:z,x,y\right)  -2S_{1,1}^{\left(  2\right)  }\left(
a,b,c:x,y,z\right) \label{29}\\
&  \ =\frac{a}{bc}\mathcal{E}_{1}\left(  bz+cy\right)  +\frac{b}%
{ac}\mathcal{E}_{1}\left(  az+cx\right)  .\nonumber
\end{align}
Invoking that $\mathcal{E}_{1}\left(  0\right)  =-1/2,$ (\ref{29}) reduce to
(cf. \cite[Theorem 5.12]{g})
\begin{equation}
S_{1,1}^{\left(  1\right)  }\left(  c,b,a\right)  -S_{1,1}^{\left(  1\right)
}\left(  c,-a,b\right)  -2S_{1,1}^{\left(  2\right)  }\left(  a,b,c\right)
=-\frac{1}{2c}\left(  \frac{a}{b}+\frac{b}{a}\right)  , \label{30}%
\end{equation}
where $S_{1,1}^{\left(  k\right)  }\left(  a,b,c\right)  =S_{1,1}^{\left(
k\right)  }\left(  a,b,c:0,0,0\right)  ,$ $k=1,2.$

It is clear that $S_{1,1}^{\left(  2\right)  }\left(  a,1,c\right)
=s_{2}\left(  a,c\right)  +1/4,$ $S_{1,1}^{\left(  1\right)  }\left(
c,1,a\right)  =s_{1}\left(  c,a\right)  -1/2,$ and $S_{1,1}^{\left(  1\right)
}\left(  c,-a,1\right)  =-1/2.$ Thus, (\ref{30}) implies (\ref{hb-12}), i.e.,%
\[
s_{1}\left(  c,a\right)  -2s_{2}\left(  a,c\right)  =\frac{1}{2}-\frac{1}%
{2}\left(  \frac{a}{c}+\frac{1}{ac}\right)  .
\]

Additionally, using that $\mathcal{B}_{1}\left(  1/2\right)  =0,$ (\ref{30})
yields
\[
S_{1,1}^{\left(  1\right)  }\left(  2,b,a\right)  -S_{1,1}^{\left(  1\right)
}\left(  2,-a,b\right)  =\frac{1}{2}-\frac{1}{4}\left(  \frac{a}{b}+\frac
{b}{a}\right)  ,
\]
or equivalently can be written as
\[
s_{1}\left(  2a^{\prime},b\right)  +s_{1}\left(  2b^{\prime},a\right)
=\frac{1}{2}-\frac{1}{4}\left(  \frac{a}{b}+\frac{b}{a}\right)  ,
\]
where $aa^{\prime}\equiv1\left(  \operatorname{mod}b\right)  ,$ $bb^{\prime
}\equiv1\left(  \operatorname{mod}a\right)  .$

Furthermore, (\ref{rp-s12}) implies a reciprocity formula for the Carlitz type
generalizations of $s_{1}\left(  a,c\right)  $ and $s_{2}\left(  a,c\right)  $
as in the following corollary.

\begin{corollary}
Let $a$ and $c$ be coprime positive integers with $c$ even. Then, for $p\geq1
$ and $x,z\in\mathbb{R}$ we have
\[
  pca^{p}S_{p}^{\left(  1\right)  }\left(  c,a:z,x\right)  -2ac^{p}%
S_{p}^{\left(  2\right)  }\left(  a,c:x,z\right)
 =\sum\limits_{j=0}^{p}\binom{p}{j}a^{p+1-j}c^{j}\mathcal{E}_{p-j}\left(
z\right)  \mathcal{B}_{j}\left(  x\right)  +p\mathcal{E}_{p}\left(
az+cx\right)  ,
\]
where $S_{p}^{\left(  k\right)  }\left(  a,c:x,z\right)  =S_{p,1}^{\left(
k\right)  }\left(  a,1,c:x,0,z\right)  ,$ $k=1,2.$
\end{corollary}

\begin{proof}
The proof follows by setting $q=b=1,$ $y=0$ in (\ref{rp-s12}) and using that%
\[
S_{p+1,0}^{\left(  1\right)  }(c,1,a:z,0,x)=\sum\limits_{\mu=0}^{a-1}%
\mathcal{E}_{p}\left(  c\frac{\mu+x}{a}+z\right)  \overset{\text{(\ref{27}%
)}}{=}a^{-p}\mathcal{E}_{p}\left(  az+cx\right)  .
\]

\end{proof}

\subsection{Generalizations of the sums $s_{3}(a,c)$, $s_{4}(a,c)$ and
$s_{5}(a,c)$}

To introduce higher-order generalizations of the sums $s_{k}(a,c)$, $k=3,4,5,
$ we set
\[
X=\frac{v+x}{da}-\frac{u+y}{db},Y=\frac{\mu+z}{dc}+\frac{u+y}{db}\text{, where
}d\text{ is even,}%
\]
in (\ref{0}) and multiply both sides with $\left(  -1\right)  ^{\mu+v}%
.$\ Summing over $v\left(  \operatorname{mod}da\right)  ,$ $u\left(
\operatorname{mod}db\right)  $ and $\mu\left(  \operatorname{mod}dc\right)  ,$
with the use of (\ref{17}) and (\ref{be}), we see that \textbf{\ }%
\begin{align*}
&  -\frac{p}{2}\left(  da\right)  ^{1-p}\left(  db\right)  ^{1-q}%
\sum\limits_{\mu=0}^{dc-1}\left(  -1\right)  ^{\mu}\mathcal{E}_{p-1}\left(
a\frac{\mu+z}{c}+x\right)  \mathcal{B}_{q}\left(  b\frac{\mu+z}{c}+y\right) \\
&  \ =\frac{q}{4}\sum\limits_{j=1}^{p}\binom{p}{j}j\frac{\left(  da\right)
^{1-j}}{\left(  dc\right)  ^{p+q-j-1}}\sum\limits_{u=0}^{db-1}\mathcal{E}%
_{p+q-j-1}\left(  c\frac{u+y}{b}+z\right)  \mathcal{E}_{j-1}\left(
-a\frac{u+y}{b}+x\right) \\
&  \quad-\frac{p}{2}\sum\limits_{h=0}^{q}\binom{q}{h}\frac{\left(  -1\right)
^{h}\left(  db\right)  ^{1-h}}{\left(  dc\right)  ^{p+q-h-1}}\sum
\limits_{v=0}^{da-1}\left(  -1\right)  ^{v}\mathcal{E}_{p+q-h-1}\left(
c\frac{v+x}{a}+z\right)  \mathcal{B}_{h}\left(  b\frac{v+x}{a}-y\right)
\end{align*}
(sums over $v$ and $\mu$ are zero for $j=0$ and $j=p+q$, respectively). We set%
\begin{align*}
S_{p,q}^{\left(  3,5\right)  }\left(  a,b,c:x,y,z\right)   &  =\sum
\limits_{\mu=0}^{c-1}\left(  -1\right)  ^{\mu}\mathcal{E}_{p-1}\left(
a\frac{\mu+z}{c}+x\right)  \mathcal{B}_{q}\left(  b\frac{\mu+z}{c}+y\right)
,\\
S_{p,q}^{\left(  4\right)  }\left(  a,b,c:x,y,z\right)   &  =\sum
\limits_{\mu=0}^{c-1}\mathcal{E}_{p-1}\left(  a\frac{\mu+z}{c}+x\right)
\mathcal{E}_{q-1}\left(  b\frac{\mu+z}{c}+y\right)
\end{align*}
since the first sum generalizes both of the sums $s_{5}\left(  a,c\right)  $
and $s_{3}\left(  b,c\right)  $ (cf. (\ref{13a}) and (\ref{13b}) below), and
the second sum generalizes $s_{4}\left(  a,c\right)  $ (cf. (\ref{13c}) below).
It is not hard to see that if $a+c$ is even, then%
\[
S_{p,q}^{\left(  3,5\right)  }\left(  da,db,dc:x,y,z\right)
=dS_{p,q}^{\left(  3,5\right)  }\left(  a,b,c:x,y,z\right)
\]
and if $a+b$ is even, then
\[
S_{p,q}^{\left(  4\right)  }\left(  da,db,dc:x,y,z\right)  =dS_{p,q}^{\left(
4\right)  }\left(  a,b,c:x,y,z\right)  .
\]
Hence, we arrive at the following linear relation.

\begin{theorem}
\label{rps543}Let $a$, $b$ and $c$ be pairwise coprime positive integers with
$a+c$ even. Then, for $p,q\geq1$ and $x,y,z\in\mathbb{R}$ we have\textbf{\ }%
\begin{align}
&  a^{1-p}b^{1-q}S_{p,q}^{\left(  3,5\right)  }\left(  a,b,c:x,y,z\right)
\label{rp-s543}\\
&  \ =-\frac{q}{2}\sum\limits_{j=1}^{p}\binom{p-1}{j-1}a^{1-j}c^{1+j-p-q}%
S_{p+q-j,j}^{\left(  4\right)  }\left(  c,-a,b:z,x,y\right) \nonumber\\
&  \quad+\sum\limits_{h=0}^{q}\binom{q}{h}\left(  -1\right)  ^{h}%
b^{1-h}c^{1+h-p-q}S_{p+q-h,h}^{\left(  3,5\right)  }\left(
c,b,a:z,-y,x\right)  .\nonumber
\end{align}

\end{theorem}

As a consequence of (\ref{rp-s543}) with
\[
S_{2,0}^{\left(  3,5\right)  }\left(  c,b,a:z,-y,x\right)  =\sum
\limits_{v=0}^{a-1}\left(  -1\right)  ^{v}\mathcal{E}_{1}\left(
\frac{cv+cx+az}{a}\right)  \overset{\text{(\ref{11})}}{=}\frac{1}%
{a}\mathcal{E}_{1}\left(  cx+az\right)  ,
\]
we have the following three-term relation%
\begin{align}
&  S_{1,1}^{\left(  3,5\right)  }\left(  a,b,c:x,y,z\right)  +\frac{1}%
{2}S_{1,1}^{\left(  4\right)  }\left(  c,-a,b:z,x,y\right)  +S_{1,1}^{\left(
3,5\right)  }\left(  c,b,a:z,-y,x\right) \label{10}\\
&  \ =\frac{b}{ac}\mathcal{E}_{1}\left(  cx+az\right)  ,\nonumber
\end{align}
a generalization of \cite[Theorem 5.6]{g}. In addition, if $b=1$ and
$x=y=z=0,$ we have $S_{1,1}^{\left(  4\right)  }\left(  c,-a,1:0,0,0\right)
=1$ and
\begin{equation}
S_{1,1}^{\left(  3,5\right)  }\left(  a,1,c:0,0,0\right)  =s_{5}\left(
a,c\right)  -\frac{1}{2}. \label{13a}%
\end{equation}
Thus, (\ref{10}) implies (\ref{hb-5}):%
\[
s_{5}\left(  a,c\right)  +s_{5}\left(  c,a\right)  =\frac{1}{2}-\frac{1}%
{2ac}.
\]

Using that
\[
S_{p+1,0}^{\left(  3,5\right)  }\left(  c,1,a:z,0,x\right)  =\sum
\limits_{\mu=0}^{a-1}\left(  -1\right)  ^{\mu}\mathcal{E}_{p}\left(
c\frac{\mu+x}{a}+z\right)  \overset{\text{(\ref{11})}}{=}a^{-p}\mathcal{E}%
_{p}\left(  az+cx\right)  ,
\]
the relation (\ref{rp-s543}) implies the following reciprocity formula.

\begin{corollary}
Let $a$ and $c$ be coprime positive odd integers. Then, for $p\geq1$ and
$x,z\in\mathbb{R}$ we have
\begin{align}
&  ac^{p}S_{p}^{\left(  5\right)  }\left(  a,c:x,z\right)  +ca^{p}%
S_{p}^{\left(  5\right)  }\left(  c,a:z,x\right) \label{rp-s5}\\
&  \ =-\frac{1}{2}\sum\limits_{j=1}^{p}\binom{p-1}{j-1}a^{p+1-j}%
c^{j}\mathcal{E}_{p-j}\left(  z\right)  \mathcal{E}_{j-1}\left(  x\right)
+\mathcal{E}_{p}\left(  az+cx\right)  ,\nonumber
\end{align}
where $S_{p}^{\left(  5\right)  }\left(  a,c:x,z\right)  =S_{p,1}^{\left(
3,5\right)  }\left(  a,1,c:x,0,z\right)  .$
\end{corollary}

Moreover, for $a=1$ and $x=y=z=0,$ we have $S_{1,1}^{\left(  3,5\right)
}\left(  c,b,1:0,0,0\right)  =\mathcal{E}_{0}\left(  0\right)  \mathcal{B}%
_{1}\left(  0\right)  =-1/2,$
\begin{equation}
S_{1,1}^{\left(  3,5\right)  }\left(  1,b,c:0,0,0\right)  =s_{3}\left(
b,c\right)  -\frac{1}{2}, \label{13b}%
\end{equation}
and%
\begin{equation}
S_{1,1}^{\left(  4\right)  }\left(  c,-1,b:0,0,0\right)  =-s_{4}\left(
c,b\right)  +1. \label{13c}%
\end{equation}
In this case, (\ref{10}) implies (\ref{hb-34}):
\[
2s_{3}\left(  b,c\right)  -s_{4}\left(  c,b\right)  =1-\frac{b}{c}.
\]

Let
\begin{align*}
S_{q}^{\left(  3\right)  }\left(  b,c:y,z\right)   &  =\sum\limits_{\mu
=0}^{c-1}\left(  -1\right)  ^{\mu}\mathcal{B}_{q}\left(  b\frac{\mu+z}%
{c}+y\right)  \mathcal{E}_{0}\left(  \frac{\mu+z}{c}\right)  ,\\
S_{q}^{\left(  4\right)  }\left(  c,b:z,y\right)   &  =\sum\limits_{\mu
=0}^{b-1}\mathcal{E}_{q-1}\left(  c\frac{\mu+y}{b}+z\right)  \mathcal{E}%
_{0}\left(  \frac{\mu+y}{b}\right)  .
\end{align*}
Then, the sums $S_{q}^{\left(  3\right)  }\left(  b,c:y,z\right)  $ and
$S_{q}^{\left(  4\right)  }\left(  c,b:z,y\right)  $ satisfy the following
reciprocity formula, which also generalizes (\ref{hb-34}).

\begin{corollary}
Let $b$ and $c$ be coprime positive integers with $c$ odd. Then, for $q\geq1$
and $y,z\in\mathbb{R}$ we have
\[
2bc^{q}S_{q}^{\left(  3\right)  }\left(  b,c:y,z\right)  -qcb^{q}%
S_{q}^{\left(  4\right)  }\left(  c,b:z,y\right)  =2\sum\limits_{h=0}%
^{q}\binom{q}{h}b^{q+1-h}c^{h}\mathcal{E}_{q-h}\left(  z\right)
\mathcal{B}_{h}\left(  y\right)  .
\]

\end{corollary}

\begin{proof}
For $p=a=1$ and $x=0,$ (\ref{rp-s543}) reduces to%
\begin{align*}
2bc^{q}S_{1,q}^{\left(  3,5\right)  }\left(  1,b,c:0,y,z\right)   &
=-qcb^{q}S_{q,1}^{\left(  4\right)  }\left(  c,-1,b:z,0,y\right) \\
&  \quad+2\sum\limits_{h=0}^{q}\binom{q}{h}\left(  -1\right)  ^{h}%
b^{q+1-h}c^{h}\mathcal{E}_{q-h}\left(  z\right)  \mathcal{B}_{h}\left(
-y\right)  .
\end{align*}
We now use (\ref{7}) and (\ref{e}) to see that
\begin{align*}
S_{q,1}^{\left(  4\right)  }\left(  c,-1,b:z,0,y\right)   &
\overset{\text{(\ref{7})}}{=}-\sum\limits_{\substack{\mu=0 \\\mu\not =\mu_{y}%
}}^{b-1}\mathcal{E}_{q-1}\left(  c\frac{\mu+y}{b}+z\right)  \mathcal{E}%
_{0}\left(  \frac{\mu+y}{b}\right) \\
&  \quad+\delta\left(  y\right)  \mathcal{E}_{q-1}\left(  c\frac{\mu_{y}+y}%
{b}+z\right)  \mathcal{E}_{0}\left(  -\frac{\mu_{y}+y}{b}\right) \\
&  =-\sum\limits_{\mu=0}^{b-1}\mathcal{E}_{q-1}\left(  c\frac{\mu+y}%
{b}+z\right)  \mathcal{E}_{0}\left(  \frac{\mu+y}{b}\right) \\
&  +\delta\left(  y\right)  \mathcal{E}_{q-1}\left(  c\frac{\mu_{y}+y}%
{b}+z\right)  \left(  \mathcal{E}_{0}\left(  \frac{\mu_{y}+y}{b}\right)
+\mathcal{E}_{0}\left(  -\frac{\mu_{y}+y}{b}\right)  \right) \\
&  \overset{\text{(\ref{e})}}{=}-S_{q}^{\left(  4\right)  }\left(
c,b:z,y\right)  +2\mathcal{E}_{q-1}\left(  z\right)  \delta\left(  y\right)  ,
\end{align*}
where $\delta\left(  y\right)  =1$ and $\mu_{y}+y\equiv0\left(
\operatorname{mod}b\right)  $ if $y\in\mathbb{Z},$ and $\delta\left(
y\right)  =0$ if $y\not \in \mathbb{Z}.$ By the similar way, it can be seen
from (\ref{8}) and the fact $\mathcal{B}_{1}\left(  0\right)  =-1/2$ that
\begin{align*}
&  2\sum\limits_{h=0}^{q}\binom{q}{h}\left(  -1\right)  ^{h}b^{q+1-h}%
c^{h}\mathcal{E}_{q-h}\left(  z\right)  \mathcal{B}_{h}\left(  -y\right) \\
&  \ \mathcal{=}2\sum\limits_{h=0}^{q}\binom{q}{h}b^{q+1-h}c^{h}%
\mathcal{E}_{q-h}\left(  z\right)  \mathcal{B}_{h}\left(  y\right)
+2qcb^{q}\mathcal{E}_{q-1}\left(  z\right)  \delta\left(  y\right)  .
\end{align*}
This completes the proof.
\end{proof}

We would like emphasize that aforementioned sums $S_{q}^{\left(  3\right)
}\left(  b,c:y,z\right)  ,$ $S_{q}^{\left(  4\right)  }(b,c:y,z)$ and
$S_{q}^{\left(  5\right)  }\left(  b,c:y,z\right)  $ are Carlitz (\ref{ct-ds})
type generalizations of the sums $s_{k}\left(  b,c\right)  ,$ $k=3,4,5,$ respectively.

\subsection{Proof of \autoref{rpS}}

For even $a+b+c$, it can be seen that
\begin{equation}
S_{p,q}\left(  da,db,dc:x,y,z\right)  =dS_{p,q}\left(  a,b,c:x,y,z\right)  .
\label{14}%
\end{equation}
To obtain (\ref{rp-S}) we set
\[
X=\frac{v+x}{da}-\frac{u+y}{db},Y=\frac{\mu+z}{dc}+\frac{u+y}{db}\text{, where
}d\text{ is even,}%
\]
in (\ref{0}) and multiply both sides with $\left(  -1\right)  ^{v+u+\mu}.$\ We
then sum up over $v\left(  \operatorname{mod}da\right)  ,$ $u\left(
\operatorname{mod}db\right)  $ and $\mu\left(  \operatorname{mod}dc\right)  ,$
with the use of (\ref{be}), and find that
\begin{align*}
&  \left(  da\right)  ^{1-p}\left(  db\right)  ^{1-q}S_{p,q}\left(
da,db,dc:x,y,z\right) \\
&  =\sum\limits_{j=1}^{p}\binom{p-1}{j-1}\left(  da\right)  ^{1-j}\left(
dc\right)  ^{1+j-p-q}S_{p+q-j,j}\left(  dc,-da,db:z,x,y\right) \\
&  \quad+\sum\limits_{h=1}^{q}\binom{q-1}{h-1}\left(  -1\right)  ^{h}\left(
db\right)  ^{1-h}\left(  dc\right)  ^{1+h-p-q}S_{p+q-h,h}\left(
dc,db,da:z,-y,x\right)
\end{align*}
(the sums over $v,$ $\mu$ and $u$ are zero for $j=0,$ $j=p+q$ and $h=0$,
respectively). Hence (\ref{rp-S}) follows from (\ref{14}).

\subsection{Hall-Wilson-Zagier type reciprocity formulas\label{rp-gf}}

In this part we shall prove a reciprocity formula for $S_{p,q}$ in terms of
the generating function. In this case, the key identity is \cite[Proposition]%
{hwz}. For simpler result, we modify the definition of $S_{p,q}$\ as%
\[
S_{p,q}\binom{a\ b\ c}{x\ y\ z}=\sum\limits_{\mu=0}^{c-1}\left(  -1\right)
^{\mu}\overline{E}_{p-1}\left(  a\frac{\mu+z}{c}-x\right)  \overline{E}%
_{q-1}\left(  b\frac{\mu+z}{c}-y\right)  ,
\]
where $\overline{E}_{p}\left(  x\right)  =\mathcal{E}_{p}\left(  x\right)  $
when $p\not =0$ or $x\not \in \mathbb{Z},$ and $\overline{E}_{0}\left(
x\right)  =0$ when $x\in\mathbb{Z}$. Following Hall, Wilson and Zagier
\cite{hwz}, we define the generating function%
\[
\Omega\left(
\begin{matrix}
a & b & c\\
x & y & z\\
X & Y & Z
\end{matrix}
\right)  =\sum\limits_{p,q\geq1}\frac{1}{p!q!}\frac{pq}{4}S_{p,q}%
\binom{a\ b\ c}{x\ y\ z}\left(  X/a\right)  ^{p-1}\left(  Y/b\right)  ^{q-1}.
\]
The following reciprocity formula holds.

\begin{theorem}
\label{g-rp}Let $a,$ $b,$ $c$ be pairwise coprime positive integers. Let $x,$
$y,$ $z\in\mathbb{R}$, and $X,$ $Y,$ $Z$ be nonzero variables such that
$X+Y+Z=0$. Then, for even $d,$%
\begin{align*}
&  \Omega\left(
\begin{matrix}
da & db & dc\\
x & y & z\\
X & Y & Z
\end{matrix}
\right)  +\Omega\left(
\begin{matrix}
dc & da & db\\
z & x & y\\
Z & X & Y
\end{matrix}
\right)  +\Omega\left(
\begin{matrix}
db & dc & da\\
y & z & x\\
Y & Z & X
\end{matrix}
\right) \\
&  \ =\left\{
\begin{array}
[c]{ll}%
\left(  -1\right)  ^{a_{0}+b_{0}+c_{0}+1}/4, & \text{if }\left(  x,y,z\right)
=\left(  da,db,dc\right)  R+\left(  a_{0},b_{0},c_{0}\right)  ,\\
0, & \text{otherwise,}%
\end{array}
\right.
\end{align*}
where $a_{0},b_{0},c_{0}\in\mathbb{Z},$ $R\in\mathbb{R}.$
\end{theorem}

\begin{proof}
Utilizing the multiplication formula
\[
-\frac{p}{2}\overline{E}_{p-1}\left(  x\right)  =r^{p-1}\sum\limits_{v=0}%
^{r-1}\left(  -1\right)  ^{v}\overline{B}_{p}\left(  \frac{x+v}{r}\right)
,\text{ }r\text{ even},
\]
the sum $S_{p,q}\binom{da\ db\ dc}{x\ \ y\ \ z}$ can be written as
\begin{align*}
&  \frac{pq}{4}\left(  da\right)  ^{1-p}\left(  db\right)  ^{1-q}S_{p,q}%
\binom{da\ db\ dc}{x\ \ y\ \ z}\\
&  \ =\sum\limits_{\mu,v,u}\left(  -1\right)  ^{\mu+v+u}\overline{B}%
_{p}\left(  \frac{\mu+z}{dc}-\frac{v+x}{da}\right)  \overline{B}_{q}\left(
\frac{\mu+z}{dc}-\frac{u+y}{db}\right)  ,
\end{align*}
where $\sum\limits_{\mu,v,u}=\sum\limits_{\substack{\mu\left(
\operatorname{mod}dc\right)  }}\sum\limits_{\substack{v\left(
\operatorname{mod}da\right)  }}\sum\limits_{\substack{u\left(
\operatorname{mod}db\right)  }}.$ \ Set
\[
w_{1}=\frac{\mu+z}{dc}-\frac{v+x}{da}\text{ and }w_{2}=\frac{u+y}{db}%
-\frac{\mu+z}{dc}.
\]
Then,%
\begin{equation}
\Omega\left(
\begin{matrix}
da & db & dc\\
x & y & z\\
X & Y & Z
\end{matrix}
\right)  =\sum\limits_{p,q\geq0}\sum\limits_{\mu,v,u}\left(  -1\right)
^{\mu+v+u}\overline{B}_{p}\left(  w_{1}\right)  \overline{B}_{q}\left(
-w_{2}\right)  \frac{X^{p-1}Y^{q-1}}{p!q!}, \label{g}%
\end{equation}
where we have used that
\begin{align*}
\sum\limits_{p,q\geq1}\overline{B}_{p}\left(  w_{1}\right)  \overline{B}%
_{q}\left(  -w_{2}\right)  \frac{X^{p-1}Y^{q-1}}{p!q!}  &  =\sum
\limits_{p,q\geq0}\overline{B}_{p}\left(  w_{1}\right)  \overline{B}%
_{q}\left(  -w_{2}\right)  \frac{X^{p-1}Y^{q-1}}{p!q!}\\
&  -\frac{1}{Y}\sum\limits_{p\geq0}\overline{B}_{p}\left(  w_{1}\right)
\frac{X^{p-1}}{p!}-\frac{1}{X}\sum\limits_{q\geq0}\overline{B}_{q}\left(
-w_{2}\right)  \frac{Y^{q-1}}{q!}%
\end{align*}
and%
\[
\sum\limits_{\mu,v,u}\left(  -1\right)  ^{\mu+v+u}\overline{B}_{p}\left(
w_{1}\right)  =\sum\limits_{\mu,v,u}\left(  -1\right)  ^{\mu+v+u}\overline
{B}_{q}\left(  -w_{2}\right)  =0.
\]
Thus, the proof follows from (\ref{g}) and \cite[Proposition]{hwz}.
\end{proof}

Similar reciprocity formulas can be derived for the sums
\begin{align*}
S_{p,q}^{\left(  1\right)  }\binom{a\ b\ c}{x\ y\ z}  &  =\sum\limits_{\mu
=0}^{c-1}\overline{E}_{p-1}\left(  a\frac{\mu+z}{c}-x\right)  \overline{B}%
_{q}\left(  b\frac{\mu+z}{c}-y\right)  ,\\
S_{p,q}^{\left(  2\right)  }\binom{a\ b\ c}{x\ y\ z}  &  =\sum\limits_{\mu
=0}^{c-1}\left(  -1\right)  ^{\mu}\overline{B}_{p}\left(  a\frac{\mu+z}%
{c}-x\right)  \overline{B}_{q}\left(  b\frac{\mu+z}{c}-y\right)
\end{align*}
and
\begin{align*}
S_{p,q}^{\left(  3,5\right)  }\binom{a\ b\ c}{x\ y\ z}  &  =\sum
\limits_{\mu=0}^{c-1}\left(  -1\right)  \overline{E}_{p-1}\left(  a\frac
{\mu+z}{c}-x\right)  \overline{B}_{q}\left(  b\frac{\mu+z}{c}-y\right)  ,\\
S_{p,q}^{\left(  4\right)  }\binom{a\ b\ c}{x\ y\ z}  &  =\sum\limits_{\mu
=0}^{c-1}\overline{E}_{p-1}\left(  a\frac{\mu+z}{c}-x\right)  \overline
{E}_{q-1}\left(  b\frac{\mu+z}{c}-y\right)  .
\end{align*}

\section{Proof of Mikol\'{a}s' linear relation\label{mlr}}

Let $a$, $b$ and $c$ be positive pairwise coprime integers. Setting
\[
X=\frac{v+x}{a}-\frac{u+y}{b},\text{ }Y=\frac{\mu+z}{c}+\frac{u+y}{b}%
\]
in (\ref{0}), then summing over $v\left(  \operatorname{mod}a\right)  ,$
$u\left(  \operatorname{mod}b\right)  $ and $\mu\left(  \operatorname{mod}%
c\right)  ,$ with $q=r+1$ and $p=m-r$, it is seen that%
\begin{align*}
&  \binom{m+1}{r+1}a^{r+1}b^{m-r}c^{m}s_{m-r,r+1}\binom{a\ b\ c}%
{x\ y\ z}-\left(  -1\right)  ^{r}c^{m+1}\mathcal{B}_{m+1}\left(  bx-ay\right)
\\
&  =\sum\limits_{j=0}^{m-r}\binom{m+1}{j}\binom{m-j}{r}a^{m+1-j}c^{j}%
b^{m}s_{m+1-j,j}\binom{c\ -a\ b}{z\ \ \ x\ \ \ y}\\
&  \quad+\sum\limits_{h=0}^{r+1}\binom{m+1}{h}\binom{m-h}{m-1-r}\left(
-1\right)  ^{h}b^{m+1-h}c^{h}a^{m}s_{m+1-h,h}\binom{c\ \ \ b\ \ \ a}%
{z\ -y\ x}.
\end{align*}
Setting $s_{m,r}\left(  a,b,c\right)  =s_{m,r}\binom{a\ b\ c}{0\ 0\ 0}$\ and
using the reflection identity (\ref{8}), we have%
\begin{align*}
&  \binom{m+1}{r+1}a^{r+1}b^{m-r}c^{m}s_{m-r,r+1}\left(  a,b,c\right) \\
&  +\sum\limits_{j=1}^{m-r}\binom{m+1}{j}\binom{m-j}{r}\left(  -1\right)
^{j+1}a^{m+1-j}c^{j}b^{m}s_{m+1-j,j}\left(  c,a,b\right) \\
&  +\sum\limits_{j=1}^{r+1}\binom{m+1}{j}\binom{m-j}{m-1-r}\left(  -1\right)
^{j+1}b^{m+1-j}c^{j}a^{m}s_{m+1-j,j}\left(  c,b,a\right) \\
&  =\left\{  \left(  -1\right)  ^{r}c^{m+1}+\binom{m}{r}a^{m+1}+\binom{m}%
{r+1}b^{m+1}\right\}  \mathcal{B}_{m+1}\left(  0\right) \\
&  \quad-\left(  m+1\right)  \binom{m-1}{r}\left(  ab\right)  ^{m}%
c\mathcal{B}_{m}\left(  0\right)  ,
\end{align*}
which is equivalent to \cite[Eq. (8)]{hw} and \cite[Eq. (5.5)]{mi}.

\end{document}